\documentclass[12pt]{amsart}
\usepackage{amsthm}
\usepackage{amssymb}
\usepackage{cite}
\usepackage{cases}
\usepackage{longtable}
\usepackage{geometry}
\usepackage{enumerate}
\usepackage{pdflscape}
\usepackage{listings}
\usepackage{setspace}
\usepackage{tikz-cd}
\usetikzlibrary{matrix, arrows}
\usepackage[unicode=true]
 {hyperref}
\hypersetup{
colorlinks=true,
urlcolor=black,
citecolor=blue,
linkcolor=blue,
}

\allowdisplaybreaks
\newtheorem{thm}{Theorem}[section]

\newtheorem{lem}[thm]{Lemma}

\newtheorem{conjecture}[thm]{Conjecture}

\theoremstyle{definition}

\theoremstyle{remark}

\numberwithin{equation}{section}

\newcommand{\bZ}{\mathbb{Z}}

\newcommand{\Perm}{\mathrm{Perm}}
\newcommand{\Hol}{\mathrm{Hol}}
\newcommand{\conj}{\mathrm{conj}}
\newcommand{\NHol}{\mathrm{NHol}}
\newcommand{\Aut}{\mathrm{Aut}}
\newcommand{\Inn}{\mathrm{Inn}}

\newcommand{\Norm}{\mathrm{Norm}}

\newcommand{\pmmod}{\hspace{-2.5mm}\pmod}

\makeatletter
\newcommand{\addresseshere}{%
  \enddoc@text\let\enddoc@text\relax
}
\makeatother

\begin{document}

\large 

\title{The multiple holomorph of centerless groups}
\author{Cindy (Sin Yi) Tsang}
\address{Department of Mathematics\\
Ochanomizu University\\
2-1-1 Otsuka, Bunkyo-ku\\
Tokyo\\
Japan}
\email{tsang.sin.yi@ocha.ac.jp}\urladdr{http://sites.google.com/site/cindysinyitsang/} 
\date{\today}

\maketitle

\vspace{-7.5mm}

\begin{abstract}Let $G$ be a group. The holomorph $\mathrm{Hol}(G)$ may be defined as the normalizer of the subgroup of either left or right translations in the group of all permutations of $G$. The multiple holomorph $\mathrm{NHol}(G)$ is in turn defined as the normalizer of the holomorph. Their quotient $T(G) = \mathrm{NHol}(G)/\mathrm{Hol}(G)$ has been computed for various families of groups $G$. In this paper, we consider the case when $G$ is centerless, and we show that $T(G)$ must have exponent at most $2$ unless $G$ satisfies some fairly strong conditions. As applications of our main theorem, we are able to show that $T(G)$ has order $2$ for all almost simple groups $G$, and that $T(G)$ has exponent at most $2$ for all centerless perfect or complete groups $G$.
\end{abstract}

\tableofcontents

\vspace{-7.5mm}

\section{Introduction}

Let $G$ be a group and write $\Perm(G)$ for the group of all permutations of $G$. Recall that a subgroup $N$ of $\Perm(G)$ is \emph{regular} if its action on $G$ is both transitive and free, or equivalently, if the map
\[ \xi_N : N \longrightarrow G;\hspace{1em}\xi_N(\eta) = \eta(1)\]
is bijective. The classical examples of regular subgroups of $\Perm(G)$ are the images of the left and right regular representations of $G$, defined by
\[ \begin{cases}
\lambda : G \longrightarrow \Perm(G);\hspace{1em}\lambda(\sigma) = (x\mapsto \sigma x),\\
\rho : G \longrightarrow \Perm(G);\hspace{1em}\rho(\sigma) = (x\mapsto x\sigma^{-1}).
\end{cases}\]
The subgroups $\lambda(G)$ and $\rho(G)$ coincide precisely when $G$ is abelian and are centralizers of each other in $\Perm(G)$. They also have the same normalizer 
\[ \Hol(G) = \Norm(\lambda(G)) = \Norm(\rho(G))\]
in $\Perm(G)$, and $\Hol(G)$ is called the \emph{holomorph} of $G$. Its normalizer
\[ \NHol(G) = \Norm(\Hol(G))\]
in $\Perm(G)$ is in turn called the \emph{multiple holomorph} of $G$. It is a known fact that isomorphic regular subgroups are conjugates (e.g. see \cite[Lemma 2.1]{Tsang ASG} for a proof). Hence, the regular subgroups of $\Perm(G)$ that are isomorphic to $G$ are exactly the conjugates of $\lambda(G)$. For any $\pi \in \Perm(G)$, we have
\[ \Norm(\pi \lambda(G) \pi^{-1}) = \pi \Norm(\lambda(G))\pi^{-1},\]
and so it follows that
\[ \pi \in \NHol(G) \iff \Norm(\pi \lambda(G)\pi^{-1}) = \Hol(G).\]
We then deduce that the quotient
\[ T(G) = \NHol(G)/\Hol(G)\] 
acts regularly via conjugation on the set 
\[\mathcal{H}_0(G) = \left\{ \begin{array}{c}
\mbox{regular subgroups $N$ of $\Perm(G)$ that are isomorphic}\\
\mbox{to $G$ and have normalizer $\Norm(N) = \Hol(G)$ in $\Perm(G)$}
\end{array}\right\}.\]
Research on the group $T(G)$ was initiated by G. A. Miller \cite{Miller} and began to attract attention again since the work of T. Kohl \cite{Kohl}. 

\vspace{2mm}

The structure of $T(G)$ has been computed for various groups $G$. In many of the known cases (e.g \cite{semisimple,Caranti1,Caranti2,Kohl,Mills,Tsang squarefree, Tsang ASG}), interestingly $T(G)$ is elementary $2$-abelian. Nonetheless, there are counterexamples. T. Kohl \cite{Kohl} found, using \textsc{gap} \cite{gap}, two groups $G$ of order $16$ for which $T(G)$ is non-abelian (although the order of $T(G)$ is still a power of $2$). By \cite{Caranti3,Tsang squarefree,Tsang metacyclic}, for odd primes $p$, there are (finite) $p$-groups $G$ of nilpotency class at most $p-1$ and split metacyclic $p$-groups $G$ for which $T(G)$ has an element of order $p$. One might wonder: Is there any non-nilpotent group $G$ for which $T(G)$ is not elementary $2$-abelian? The answer is ``yes". For example, as was mentioned in \cite[(1.5)]{Tsang squarefree}, the order of $T(G)$ is not even a power of $2$ for $G = \textsc{SmallGroup}(a,b)$ with
\[ (a,b) = (48,12),(48,14),(63,1),(80,12),(80,14),\]
which were found using \textsc{Magma} \cite{magma}. The present author \cite[Section 4]{Tsang ADM} also described a method to construct groups of the form
\[ G = A \rtimes (\bZ/p^n\bZ),\,\ n\geq 2,\]
where $A$ is an abelian group of finite exponent coprime to $p$, such that $T(G)$ has an element of order $p$. For the argument to work, one needs $p^{n-1} + p^n\bZ$ to act trivially on $A$. One then finds that all of these counterexamples $G$ have non-trivial center and are solvable groups. 

\vspace{2mm}

In the study of $T(G)$, most of the groups $G$ considered are finite, in which case of course $T(G)$ is also finite. But for infinite groups $G$, the author does not know whether $T(G)$ has to be finite or not. To take the infinite case into account, instead of elementary $2$-abelian, we therefore ask whether $T(G)$ has exponent at most $2$. In view of the known results so far, it seems reasonable to make the following conjectures.

\begin{conjecture}\label{conj1}The quotient $T(G)$ has exponent at most $2$ for all (finite) centerless groups $G$.
\end{conjecture}

\begin{conjecture}\label{conj2}The quotient $T(G)$ has exponent at most $2$ for all (finite) insolvable groups $G$ which does not have any solvable normal subgroup as a direct factor. 
\end{conjecture}

In this paper, we shall study Conjecture \ref{conj1} without the assumption that $G$ is finite. In order to state our main theorem in full generality, we would have to set up some notation. We shall therefore postpone its statement to Theorem \ref{main thm}. Here, let us just state that our result implies that $G$ needs to satisfy some fairly strong conditions when $T(G)$ has exponent greater than $2$ (Theorem \ref{thm}). As applications, we are able to prove Conjecture \ref{conj1} in some special cases (Theorems \ref{thm1}, \ref{thm2}, \ref{thm3}, \ref{thm4}). In particular:
\begin{enumerate}[(1)]
\item $T(G)$ is cyclic of order $2$ for all almost simple groups $G$;
\item  $T(G)$ has exponent at most $2$ for all centerless perfect groups $G$;
\item  $T(G)$ has exponent at most $2$ for all complete groups $G$;
\item $T(G)$ has exponent at most $2$ for all centerless groups $G$ of order $\leq2000$, except possibly when $G$ has order $1536$ or when $G$ has \textsc{SmallGroup} ID  equal to $(605,5),(1210,11)$.
\end{enumerate}
Let us remark that $(1)$ and $(2)$ were known for finite groups $G$ by \cite{Tsang ASG,Caranti2}. But here $G$ could be infinite, and unlike \cite{Tsang ASG} our proof of (1) does not require the classification of finite simple groups. Also $(4)$ was verified using \textsc{Magma} \cite{magma}, and we had to exclude the order $1536$ because there are over $4\times 10^8$ groups of order $1536$ and the code simply takes too long to run.

\section{Preliminaries}\label{prelim sec}

In this section, let $G$ be an arbitrary group, not necessarily centerless. Let us fix an element $\pi \in \NHol(G)$, and we wish to understand when its class in $T(G)$ has order at most $2$, in other words when $\pi^2\in \Hol(G)$ holds. We shall study $\pi$ via a triplet $(f,h,g)$, to be defined below, where $f,h:G\longrightarrow\Aut(G)$ are homomorphisms and $g:G\longrightarrow G$ is a bijection.

\vspace{2mm}

Put $N =\pi \lambda(G)\pi^{-1}$, which is plainly a regular subgroup of $\Perm(G)$ that is isomorphic to $G$. The fact that $\pi \in \NHol(G)$ implies
\begin{align*}
N= \pi \lambda(G) \pi^{-1} &\subseteq \pi \Hol(G)\pi^{-1} = \Hol(G),\\
\Norm(N) = \pi \Norm(\lambda(G))\pi^{-1} &= \pi \Hol(G)\pi^{-1} = \Hol(G).
\end{align*}
It follows that $N$ is in fact a normal subgroup of $\Hol(G)$. Let us remark in passing that the above shows that
\[ \mathcal{H}_0(G) \subseteq \{\mbox{normal regular subgroups of $\Hol(G)$ isomorphic to $G$}\}\]
for the set $\mathcal{H}_0(G)$ defined in the introduction. This inclusion is easily shown to be an equality when $G$ is finite (see \cite[p. 954]{Tsang ASG}), but the author does not know whether it remains true when $G$ is infinite.

\vspace{2mm}

Now, it is known and easily verified that
\begin{equation}\label{Hol}
 \Hol(G) = \lambda(G) \rtimes \Aut(G)=\rho(G) \rtimes \Aut(G).
\end{equation}
Note that $N$ is isomorphic to $G$ via $\sigma\mapsto \pi\lambda(\sigma)\pi^{-1}$. By projecting $\pi\lambda(\sigma)\pi^{-1}$ onto the two components in the latter semidirect product decomposition, we may view $N$ as the image of an injective homomorphism of the form
\[ \beta : G \longrightarrow \Hol(G);\hspace{1em}\beta(\sigma) = \rho(g(\sigma)) f(\sigma),\]
where $f\in \mathrm{Hom}(G,\Aut(G))$ and $g\in \mathrm{Map}(G,G)$. Explicitly, we have
\[ N = \{ \rho(g(\sigma)) f(\sigma) : \sigma \in G\}.\]
It is straightforward to check that $\beta$ being an homomorphism implies that
\begin{equation}\label{relation} 
g(\sigma\tau) = g(\sigma) f(\sigma)(g(\tau)) \mbox{ for all }\sigma,\tau\in G,
\end{equation}
and $N$ being a regular subgroup implies that $g$ is bijective (or see  \cite[Proposition 2.1]{Tsang HG} for a proof). Note that clearly $g(1) =1$.

\vspace{2mm}

Let $\Inn(G)$ denote the inner automorphism group of $G$, and write
\[ \conj : G\longrightarrow \Inn(G);\hspace{1em} \conj(\sigma) = (x\mapsto \sigma x\sigma^{-1})\]
for the natural homomorphism. Define
\begin{equation}\label{h def}
h\in \mathrm{Hom}(G,\Aut(G));\hspace{1em}h(\sigma) = \conj(g(\sigma)) f(\sigma).
\end{equation}
A simple calculation using (\ref{relation}) shows that $h$ is indeed a homomorphism (or see \cite[Proposition 3.4]{Tsang ASG} for a proof). We note that then
\[ N = \{\lambda(g(\sigma)^{-1}) h(\sigma) : \sigma \in G\},\]
corresponding to the first semidirect product decomposition in (\ref{Hol}). Also
\begin{equation}\label{relation'}
g(\sigma\tau) = h(\sigma)(g(\tau))g(\sigma) \mbox{ for all }\sigma,\tau\in G,
\end{equation}
which follows immediately from (\ref{relation}).

\vspace{2mm}

That $N$ is a normal subgroup of $\Hol(G)$ implies the subgroups $f(G),h(G)$ of $\Aut(G)$ are also normal. The next lemma shows that in fact they are in some sense ``simultaneously normal".

\begin{lem}\label{normal lemma}For any $\varphi \in \Aut(G)$ and $\sigma\in G$, we have
\[ \varphi f(\sigma) \varphi^{-1} = f(\sigma_\varphi) \hspace{2mm}
\mbox{and}\hspace{2mm}\varphi h(\sigma)\varphi^{-1} = h(\sigma_\varphi) \]
for the unique $\sigma_\varphi \in G$ such that $\varphi(g(\sigma)) = g(\sigma_\varphi)$.
\end{lem}
\begin{proof}Simply observe that
\begin{align*}
 \varphi \cdot \rho(g(\sigma)) f(\sigma) \cdot \varphi^{-1} & = \rho(g(\sigma_\varphi)) \cdot \varphi f(\sigma)\varphi^{-1},\\
 \varphi \cdot \lambda(g(\sigma)^{-1})h(\sigma)\cdot \varphi^{-1} & = \lambda(g(\sigma_\varphi)^{-1}) \cdot \varphi h(\sigma)\varphi^{-1}.
\end{align*}
That $N$ is normal in $\Hol(G)$ implies that these are elements of $N$, which is equivalent to the two equalities stated in the lemma since $g$ is bijective.
\end{proof}

Let us also explain how the two permutations $\pi$ and $g$ are related. Write
\[ \iota : G\longrightarrow G;\hspace{1em}\iota(\sigma) =\sigma^{-1}\]
for the inversion map. We have $\iota \in \NHol(G)$ and $\rho(G) = \iota\lambda(G)\iota^{-1}$. Its class in $T(G)$ is trivial for $G$ abelian and has order $2$ for $G$ non-abelian. 

\begin{lem}\label{g lemma} Put $g_0 = \iota\circ g$. Then $N = g_0\lambda(G)g_0^{-1}$ and we have
\[ \pi \equiv g_0 \pmod{\Hol(G)}.\]
In particular, by $(\ref{relation})$ and $(\ref{relation'})$, this implies that
\[\pi \equiv \ \begin{cases}\iota\pmmod{\Hol(G)}&\mbox{if $f(G)=1$ (so $g_0$ is an anti-homomorphism)},\\
1\pmmod{\Hol(G)}&\mbox{if $h(G)=1$ (so $g_0$ is a homomorphism)}.
\end{cases}
\]
\end{lem}
\begin{proof}For any $\sigma,x\in G$, from the relation (\ref{relation}), we see that
\begin{align*}
(g_0\lambda(\sigma)g_0^{-1})(g(x)^{-1})
& = g(\sigma x)^{-1}\\
& = (g(\sigma) f(\sigma)(g(x)))^{-1} \\
& = f(\sigma)(g(x)^{-1}) g(\sigma)^{-1}\\
& = (\rho(g(\sigma))f(\sigma))(g(x)^{-1}).
\end{align*}
Since $g$ is bijective and this holds for all $x\in G$, this implies that
\[  g_0\lambda(\sigma) g_0^{-1}= \rho(g(\sigma)) f(\sigma)\mbox{ for all }\sigma\in G\]
whence the first claim. The second claim follows because $N = \pi\lambda(G)\pi^{-1}$ by definition and we know that $\Norm(N)= \Hol(G)$. 
\end{proof}

\section{Our main theorem}\label{statement sec}

In this section, we shall assume that $G$ is centerless. Let $\pi \in \NHol(G)$ and let $f,h\in\mathrm{Hom}(G,\Aut(G)),\, g\in\Perm(G)$ be defined as in Section \ref{prelim sec}. 

\vspace{2mm}

Before stating our main theorem, let us just explain why the approach to study $\pi$ via $(f,h,g)$ works particularly well for centerless groups $G$. First of all, since $g$ is not a homomorphism unless $f$ is trivial, it could be difficult to understand. But $\conj$ is now an isomorphism, so by (\ref{h def}) we can pass to the pair $(f,h)$ of homomorphisms without losing any information. Second, that
\[ \varphi \conj(\sigma) \varphi^{-1}\conj(\sigma)^{-1} = \conj(\varphi(\sigma)\sigma^{-1}) \mbox{ for all }\varphi\in \Aut(G),\, \sigma\in G\]
implies that any subgroup of $\Aut(G)$ containing $\Inn(G)$ is also centerless.

\vspace{2mm}

With these two simple yet useful facts, we are able to prove that $(f,h,g)$ has to satisfy some fairly strong conditions, hence reducing the possibilities for $\pi$ significantly. The list of conditions is long, but for the convenience of the interested reader who might want to apply our theorem to make further progress on Conjecture \ref{conj1}, we included as much information as possible in the statement. We note that the most crucial conditions are (2) and (3).

\begin{thm}\label{main thm}Let $G$ be a centerless group. In the above notation:
\begin{enumerate}[A.]
\item All of the following  conditions are satisfied.
\begin{enumerate}[$(1)$]
\item $\ker(f)\cap\ker(h)=1$.
\item $[G,G]\subseteq\ker(f)\ker(h)$.
\item $[f(G),h(G)]=1$.
\item $\Inn(G)\subseteq f(G)h(G)$ and $f(G)\Inn(G) = h(G)\Inn(G)$.
\item $f(G)\cap h(G)=1$.
\item $f(G),h(G)$ are centerless normal subgroups of $\Aut(G)$.
\item $f(G)\cap \Inn(G)=f(\ker(h))$.
\item $h(G)\cap \Inn(G)=h(\ker(f))$.
\item $\{f(\sigma)h(\sigma):\sigma\in G\}$ is normal subgroup of $\Aut(G)$ isomorphic to $G$.
\item $g(\ker(f)),g(\ker(h))$ are characteristic subgroups of $G$.
\item $g(\ker(f))\simeq \ker(f)$ and $g(\ker(h))\simeq \ker(h)$.
\item $g$ induces a surjective homomorphism $G\longrightarrow G/g(\ker(h))$.
\item $g$ induces a surjective anti-homorphism $G\longrightarrow G/g(\ker(f))$.
\item $G/\ker(h)\simeq G/g(\ker(h))$ and $G/\ker(f)\simeq G/g(\ker(f))$.
\item $G/\ker(f)\ker(h)\simeq G/g(\ker(f))g(\ker(h))$.
\end{enumerate}
\vspace{2mm}
\item We have $\pi^2\in \Hol(G)$ if and only if both $(a)$ and $(b)$ below are satisfied.
\begin{enumerate}[$(a)$]
\item $g$ induces an anti-homomorphism from $G$ to $G/\ker(f)$.
\item $g$ induces a homomorphism from $G$ to $G/\ker(h)$.
\end{enumerate}
In particular, by $(12),(13)$, we have $\pi^2\in\Hol(G)$ if $(c)$ below holds.
\begin{enumerate}[$(a)$]\setcounter{enumii}{+2}
\item$g(\ker(f))\subseteq \ker(f)$ and  $g(\ker(h))\subseteq \ker(h)$.
\end{enumerate}
\vspace{2mm}
\item The following conditions are equivalent and they imply $(c)$ above.
\begin{enumerate}[$(a)$]\setcounter{enumii}{+3}
\item $\Inn(G)=\{f(\sigma)h(\sigma):\sigma\in G\}$.
\item $G/\ker(f)\ker(h)$ has exponent at most $2$.
\end{enumerate}
\end{enumerate}
\end{thm}


\subsection{Proof of A} Write $[x,y]=xyx^{-1}y^{-1}$ for any $x,y$ in a given group.

\begin{proof}[Proof of $(1)$] For any $\sigma\in \ker(f)\cap \ker(h)$, from (\ref{h def}) we have $\conj(g(\sigma))=1$. But $\conj$ is an isomorphism and $g$ is bijective, whence $\sigma=1$.
\end{proof}

\begin{proof}[Proof of $(2),(3)$] Let $\sigma,\tau\in G$ be arbitrary. By Lemma \ref{normal lemma}, we have
\begin{align*}
f(\sigma) f(\tau) f(\sigma)^{-1}&= f( \tau_\sigma)\\
f(\sigma) h(\tau) f(\sigma)^{-1}& = h(\tau_\sigma)
\end{align*}
for some $\tau_\sigma\in G$, and similarly
\begin{align*}
h(\tau)f(\sigma)h(\tau)^{-1} & = f(\sigma_\tau)\\
h(\tau)h(\sigma)h(\tau)^{-1} & = h(\sigma_\tau)
\end{align*}
for some $\sigma_\tau\in G$. From the first and last equalities, we deduce that
\begin{align*}
\tau_\sigma & = \zeta_f\sigma\tau\sigma^{-1}\mbox{ for some }\zeta_f\in \ker(f),\\ 
\sigma_\tau & = \zeta_h \tau\sigma\tau^{-1}\mbox{ for some }\zeta_h\in \ker(h).
\end{align*}
Plugging them into the middle two equalities then yields
\begin{align*}
f(\sigma)h(\tau)f(\sigma)^{-1} & = h(\zeta_f\sigma\tau\sigma^{-1}),\mbox{ whence } [f(\sigma),h(\tau)]=h(\zeta_f[\sigma,\tau]),\\
h(\tau)f(\sigma)h(\tau)^{-1} & = f(\zeta_h\tau\sigma\tau^{-1}),\mbox{ whence }[h(\tau),f(\sigma)] = f(\zeta_h[\tau,\sigma]).
\end{align*}
They in turn imply that
\[ h(\zeta_f[\sigma,\tau]) = f(\zeta_h[\tau,\sigma])^{-1}\mbox{ and so } h(\zeta_f[\sigma,\tau]\zeta_h^{-1}) = f(\zeta_f[\sigma,\tau]\zeta_h^{-1}).\]
But from (\ref{h def}) this means that
\[ \conj(g(\zeta_f[\sigma,\tau]\zeta_h^{-1})) = 1.\]
Since $\conj$ is an isomorphism and $g$ is bijective, it follows that
\[ \zeta_f[\sigma,\tau]\zeta_h^{-1}=1,\mbox{ or equivalently } [\sigma,\tau] = \zeta_f^{-1}\zeta_h,\]
which proves (2). It also implies that
\[ [f(\sigma),h(\tau)] = h(\zeta_f[\sigma,\tau]) = h(\zeta_f\zeta_f^{-1}\zeta_h) = 1,\]
which proves (3).
\end{proof}

\begin{proof}[Proof of $(4),(5),(6)$] Since $g$ is bijective, it is clear from (\ref{h def}) that (4) holds. It then follows that $f(G)h(G)$ is centerless. Since $f(G)\cap h(G)$ has to lie in the center of $f(G)h(G)$ by (3), it must be trivial and this yields (5). It also implies that $f(G)h(G) = f(G)\times h(G)$, which is centerless and so we deduce that $f(G),h(G)$ are both centerless. That $f(G),h(G)$ are normal in $\Aut(G)$ is already known by Lemma \ref{normal lemma}, whence (6) holds.
\end{proof}

\begin{proof}[Proof of $(7),(8)$] For any $\sigma\in G$, we have by (\ref{h def}) that $f(\sigma)\in \Inn(G)$ if and only if $h(\sigma)\in \Inn(G)$. Since $g$ is bijective, in this case $f(\sigma) = \conj(g(\tau))$ for some $\tau\in G$. But again by (\ref{h def}), this means that
\[ f(\sigma\tau)= f(\sigma)f(\tau) = \conj(g(\tau))f(\tau) = h(\tau).\]
By (5), this in turn implies that
\[ \sigma\tau\in \ker(f),\, \tau\in\ker(h),\mbox{ whence }
\sigma = (\sigma\tau)\tau^{-1} \in \ker(f)\ker(h).\]
We then obtain the inclusions
\begin{align*}
 f(G)\cap \Inn(G) &\subseteq f(\ker(f)\ker(h)) = f(\ker(h)),\\
h(G)\cap\Inn(G) & \subseteq h(\ker(f)\ker(h)) = h(\ker(f)).
\end{align*}
The reverse inclusions are obvious from (\ref{h def}), whence (7),(8) hold.
\end{proof}

\begin{proof}[Proof of $(9)$] Consider the map
\[ \Phi : G \longrightarrow \Aut(G);\hspace{1em}\Phi(\sigma) = f(\sigma)h(\sigma),\]
which is a homomorphism by (3). Its image, which is the subset in question, is then a subgroup of $\Aut(G)$ and is normal by Lemma \ref{normal lemma}. Its kernel lies in $\ker(f)\cap \ker(h)$ by (5), and so is trivial by (1). It follows that $G\simeq \Phi(G)$ and we obtain (9).
\end{proof}

\begin{proof}[Proof of $(10),(11)$] By (\ref{relation}) and (\ref{relation'}), the map $g$ is a homomorphism and an anti-homomorphism when restricted to $\ker(f)$ and $\ker(h)$, respectively. Since $g$ is a bijection, it follows that $g(\ker(f))$ and $g(\ker(h))$ are indeed subgroups of $G$, which are isomorphic to $\ker(f)$ and $\ker(h)$, respectively.

\vspace{2mm}

Next, consider $\varphi\in \Aut(G)$. For any $\zeta\in G$, by Lemma \ref{normal lemma}, we know that
\[ \varphi f(\zeta) \varphi^{-1} = f(\zeta_\varphi) \hspace{2mm}
\mbox{and}\hspace{2mm}\varphi h(\zeta)\varphi^{-1} = h(\zeta_\varphi) \]
for the unique $\zeta_\varphi \in G$ such that $\varphi(g(\zeta)) = g(\zeta_\varphi)$. From these two equalities, respectively, we deduce that
\[ \begin{cases}
\zeta_\varphi\in\ker(f)\mbox{ and so }\varphi(g(\zeta))\in g(\ker(f))&\mbox{for all }\zeta\in \ker(f),\\
\zeta_\varphi\in\ker(h)\mbox{ and so }\varphi(g(\zeta))\in g(\ker(h))&\mbox{for all }\zeta\in \ker(h).
\end{cases}\]
 This shows that $g(\ker(f))$ and $g(\ker(h))$ are characteristic subgroups of $G$, thus proving (10) and (11).
\end{proof}

\begin{proof}[Proof of $(12),(13)$] For any $\sigma,\tau\in G$, we already know by (3) that
\[ [f(\sigma),h(\tau)]=1\mbox{ and similarly }[h(\sigma),f(\tau)]=1.\]
Using (\ref{h def}), we may express these equalities as
\begin{align*}
f(\sigma) \conj(g(\tau))f(\tau) f(\sigma)^{-1} f(\tau)^{-1}\conj(g(\tau))^{-1} &= 1,\\
 h(\sigma)\conj(g(\tau))^{-1}h(\tau) h(\sigma)^{-1} h(\tau)^{-1}\conj(g(\tau)) & = 1.
 \end{align*}
Let us rearrange these equalities and write them as
\begin{align*}
\conj(f(\sigma)(g(\tau))^{-1}g(\tau)) & = f([\sigma,\tau]),\\
\conj(h(\sigma)(g(\tau))g(\tau)^{-1}) & = h([\sigma,\tau]).
\end{align*}
Now, we know from (2) that
\[ [\sigma,\tau] = \zeta_f\zeta_h\mbox{ for some }\zeta_f\in \ker(f),\, \zeta_h\in \ker(h).\]
Using (\ref{h def}) again, we then obtain
\begin{align*}
\conj(f(\sigma)(g(\tau))^{-1}g(\tau)) & = f(\zeta_f\zeta_h) = f(\zeta_h) = \conj(g(\zeta_h)^{-1}),\\
\conj(h(\sigma)(g(\tau))g(\tau)^{-1}) & = h(\zeta_h\zeta_f) = h(\zeta_f) = \conj(g(\zeta_f)).
\end{align*}
Since $\conj$ is an isomorphism, this implies that
\begin{align*}
f(\sigma)(g(\tau)) &\equiv g(\tau) \pmod{g(\ker(h))},\\
h(\sigma)(g(\tau)) & \equiv g(\tau) \pmod{g(\ker(f))}.
\end{align*}
By (\ref{relation}) and (\ref{relation'}), these in turn yield
\begin{align*}
g(\sigma\tau) &\equiv g(\sigma)g(\tau) \pmod{g(\ker(h))},\\
g(\sigma\tau) & \equiv g(\tau)g(\sigma) \pmod{g(\ker(f))}.
\end{align*}
Hence, indeed $g$ induces a homomorphism $G\longrightarrow G/g(\ker(h))$, and also an anti-homomorphism $G\longrightarrow G/g(\ker(f))$. Their surjectivity  is clear because $g$ is bijective. This proves (12),(13).\end{proof}

\begin{proof}[Proof of $(14),(15)$] By (12),(13), we have surjective homomorphisms
\[\begin{cases}
g_h : G\longrightarrow G/g(\ker(h));&\hspace{1em} g_h(\sigma) = g(\sigma)g(\ker(h)), \\[3pt]
g_f' : G \longrightarrow G/g(\ker(f));&\hspace{1em} g_f'(\sigma) = g(\sigma)^{-1}g(\ker(f)),\\[3pt]
 g_{fh} : G\longrightarrow G/g(\ker(f))g(\ker(h));&\hspace{1em}g_{fh}(\sigma) = g(\sigma)g(\ker(f))g(\ker(h)).
\end{cases}\]
Since $g$ is bijective, we have
\[\ker(g_h) = \ker(h)\mbox{ and }\ker(g_f') = \ker(f),\]
whence (14) holds. Observe that by (\ref{relation}), we have
\[ g(\zeta_f)g(\zeta_h) = g(\zeta_f\zeta_h)\mbox{ for all }\zeta_f\in \ker(f),\, \zeta_h\in \ker(h).\]
This means that we have the equality
\[ g(\ker(f))g(\ker(h)) = g(\ker(f)\ker(h)).\]
Since $g$ is bijective, we see that 
\[\ker(g_{fh}) = \ker(f)\ker(h),\]
so (16) follows as well.
\end{proof}

\subsection{Proof of B} 

Put $g_0 = \iota\circ g$ as in Lemma \ref{g lemma}. Then, we have
\[ \pi^2 \in \Hol(G) \iff g_0^2 \in \Hol(G)
\iff g_0^2 \in \Aut(G),\]
where the second equivalence follows from (\ref{Hol}) and the fact that $g_0(1)=1$. Since $g_0$ is a bijection, this means that $ \pi^2 \in \Hol(G)$ precisely when $g_0^2$ is a homomorphism.

\vspace{2mm}

For any $\sigma,\tau\in G$, observe that by (\ref{h def}), we have
\begin{align*}\notag
\conj(g_0^2(\sigma)) & = \conj(g(g_0(\sigma)))^{-1} = f(g_0(\sigma))h(g_0(\sigma))^{-1},\\\notag
\conj(g_0^2(\tau)) & = \conj(g(g_0(\tau)))^{-1} = f(g_0(\tau))h(g_0(\tau))^{-1}.
\end{align*}
Since $f(G)$ and $h(G)$ commute elementwise by (3),  we obtain
\[ \conj(g_0^2(\sigma)g_0^2(\tau)) = f(g_0(\sigma)g_0(\tau)) h(g_0(\tau)g_0(\sigma))^{-1}.\]
But on the other hand, again by (\ref{h def}), we have
\[\conj(g_0^2(\sigma\tau))  = \conj(g(g_0(\sigma\tau)))^{-1} = f(g_0(\sigma\tau))h(g_0(\sigma\tau))^{-1}.\]
Since $\conj$ is an isomorphism and $f(G)\cap h(G)$ is trivial by (5), by comparing the above expressions we see that
\begin{align*}
g_0^2(\sigma\tau) = g_0^2(\sigma)g_0^2(\tau)
&\iff \begin{cases}
g_0(\sigma\tau) \equiv g_0(\sigma)g_0(\tau)\pmmod{\ker(f)},\\
g_0(\sigma\tau) \equiv g_0(\tau)g_0(\sigma)\pmmod{\ker(h)},\end{cases}\\[3pt]
&\iff \begin{cases}
g(\sigma\tau) \equiv g(\tau)g(\sigma) \pmmod{\ker(f)},\\
g(\sigma\tau) \equiv g(\sigma)g(\tau)\pmmod{\ker(h)}.
\end{cases} \end{align*}
These are exactly the conditions (a) and (b), respectively. 

\subsection{Proof of C}

Observe that by (\ref{h def}) and (7), for any $\sigma\in G$ we have 
\begin{equation}\label{iff}
f(\sigma)h(\sigma) \in \Inn(G) \iff f(\sigma^2)\in \Inn(G)\iff \sigma^2\in \ker(f)\ker(h).\end{equation}
 It is then clear that (d) implies (e).  Conversely, suppose that (e) holds. That
\[ \{f(\sigma)h(\sigma): \sigma\in G\} \subseteq \Inn(G)\]
follows immediately from (\ref{iff}). To show the reverse inclusion, note that 
\[ \Inn(G) = \{h(\sigma)f(\sigma)^{-1}: \sigma\in G\}\]
by (\ref{h def}) and the bijectivity of $g$. For any $\sigma\in G$, we may write
\[ \sigma^2 = \zeta_f\zeta_h\mbox{ for some }\zeta_f\in \ker(f),\, \zeta_h\in \ker(h)\]
by the hypothesis (e). Together with (3), it follows that
\[
h(\sigma)f(\sigma)^{-1}  = f(\sigma^2)^{-1} f(\sigma)h(\sigma)
 = f(\zeta_h^{-1}\sigma) h(\zeta_h^{-1}\sigma).
\]
We then obtain the inclusion
\[ \Inn(G) \subseteq \{f(\sigma)h(\sigma):\sigma\in G\}\]
and hence (d) holds. This proves the equivalence of (d) and (e).

\vspace{2mm}

Next, suppose that (d) is satisfied and recall from the proof of (9) that
\[ \Phi : G\longrightarrow \{f(\sigma)h(\sigma):\sigma\in G\};\hspace{1em}\Phi(\sigma) =f(\sigma)h(\sigma)\]
is an isomorphism. Since $\Phi(G) = \Inn(G)$ by the assumption (d) and $\conj$ is also an isomorphism, we may define 
\[\varphi \in \Aut(G);\hspace{1em} \varphi = \Phi^{-1}\circ \conj.\]
For any $\zeta_f\in \ker(f)$ and $\zeta_h\in \ker(h)$, by (\ref{h def}) we know that
\begin{align*}
\conj(g(\zeta_f)) & = h(\zeta_f) = f(\zeta_f)h(\zeta_f), \mbox{ whence }\varphi(g(\zeta_f)) =\zeta_f.\\
\conj(g(\zeta_h)) & = f(\zeta_h)^{-1} = f(\zeta_h^{-1})h(\zeta_h^{-1}),\mbox{ whence }\varphi(g(\zeta_h)) = \zeta_h^{-1}. 
\end{align*}
This implies that
\[ \varphi(g(\ker(f))) = \ker(f)\mbox{ and }\varphi(g(\ker(h)))= \ker(h).\]
But $g(\ker(f))$ and $g(\ker(h))$ are characteristic subgroups of $G$ by (10),(11), so they are in fact equal to $\ker(f)$ and $\ker(h)$, respectively. We then deduce that condition (c) is satisfied.

\section{Applications}\label{app sec}

Although we were unable to prove Conjecture \ref{conj1} in full, our Theorem \ref{main thm} implies that any counterexample $G$, if exists, must satisfy some fairly strong conditions as stated below. Here two normal subgroups $K_1$ and $K_2$ of $G$ are said to be \emph{series-equivalent} if $K_1\simeq K_2$ and $G/K_1\simeq G/K_2$.

\begin{thm}\label{thm}Let $G$ be any non-trivial centerless group.
\begin{enumerate}[$(i)$]
\item If $T(G)$ is not cyclic of order $2$, then there exist
\begin{itemize}
\item non-trivial proper characteristic subgroups $K_1,K_2$ of $G$,
\item centerless normal subgroups $Q_1,Q_2$ of $\Aut(G)$,
\end{itemize}
such that all of the following conditions are satisfied.
\begin{itemize}
\item $K_1\cap K_2=1$ and $[G,G]\subseteq K_1K_2$.
\item $Q_1\cap Q_2=1$, $\Inn(G)\subseteq Q_1Q_2$, $Q_1\Inn(G)=Q_2\Inn(G)$.
\item $G/K_1\simeq Q_1$ and $G/K_2\simeq Q_2$.
\end{itemize}
\vspace{2mm}
\item If $T(G)$ has exponent greater than $2$, then the $K_1,K_2,Q_1,Q_2$ in $(i)$ may be chosen to satisfy all of the following conditions.
\begin{itemize}
\item $G/K_1K_2$ has exponent greater than $2$.
\item $K_1$ is series-equivalent to a normal subgroup $K$ of $G$ with $K_1\not\subseteq K$.
\item $Q_1Q_2$ contains a normal subgroup $\Gamma$ of $\Aut(G)$ which is isomorphic to $G$ but not equal to $\Inn(G)$.
\item $Q_1,Q_2$ are neither contained in $\Inn(G)$.
\end{itemize}
\end{enumerate}
\end{thm}
\begin{proof}Let $\pi \in \NHol(G)$ and let $(f,h,g)$ be defined as in Section \ref{prelim sec}. 

\vspace{2mm}

Let $\iota$ be as in Lemma \ref{g lemma} and notice that  the classes $\Hol(G)$ and $\iota\Hol(G)$ are distinct since $G$ is non-abelian. Also, by Lemma \ref{g lemma}, they correspond to when $h(G)=1$ and $f(G)=1$, respectively.

\vspace{2mm}

Suppose that $T(G)$ is not cyclic of order $2$ so we can pick $\pi\in \NHol(G)$ to be such that $\ker(f)$ and $\ker(h)$ are both proper subgroups of $G$. Moreover, by Theorem \ref{main thm}, if $\ker(f)=1$, then $[G,G]\subseteq \ker(h)$, and $G/\ker(h) \simeq h(G)$ would be both abelian and centerless, which is impossible because $\ker(h)$ is proper. If $\ker(h)=1$, then we obtain a contradiction similarly. We conclude that $\ker(f)$ and $\ker(h)$ are also non-trivial. Since $g$ is a bijection, we deduce that the subgroups $g(\ker(f))$ and $g(\ker(h))$ are also non-trivial and proper. We may then take $(K_1,K_2,Q_1,Q_2)$ to be either of the following
\begin{align}
    (g(\ker(f)),g(\ker(h)),f(G),h(G)) \label{choice1}
   \\
   (g(\ker(h)),g(\ker(f)),h(G),f(G))  \label{choice2}
\end{align}
so that (i) holds. Notice that $[G,G]\subseteq K_1K_2$ by (2),(15) of Theorem \ref{main thm}. All of the other required conditions are clear from Theorem \ref{main thm}.
 
\vspace{2mm}

Suppose now that $T(G)$ has exponent greater than $2$ so we can pick $\pi\in \NHol(G)$ to be such that $\pi^2\not\in \Hol(G)$. By Theorem \ref{main thm}, one of
\[ g(\ker(f)) \subseteq \ker(f)\mbox{ and }g(\ker(h))\subseteq \ker(h)\]
has to fail, and we may take
\[ \begin{cases}
(K_1,K_2,Q_1,Q_2) = (\ref{choice1}),\,\ K = \ker(f) &\mbox{if }g(\ker(f))\not\subseteq \ker(f),\\
(K_1,K_2,Q_1,Q_2) = (\ref{choice2}),\,\ K = \ker(h) &\mbox{otherwise},
\end{cases}\]
so that $K_1\not\subseteq K$ but $K_1$ and $K$ are series-equivalent. Again by Theorem \ref{main thm}, we know that $G/K_1K_2$ has exponent greater than $2$, and we may take
\[ \Gamma = \{f(\sigma)h(\sigma): \sigma\in G\}.\]
Moreover, if either of $Q_1,Q_2$ were to lie in $\Inn(G)$, then (7),(8),(15) of Theorem \ref{main thm} would imply that
\[ G = \ker(f)\ker(h)\mbox{ and so } G/K_1K_2\simeq G/\ker(f)\ker(h)\]
is trivial, which is a contradiction. This proves (ii).
\end{proof}

We shall end by applying Theorem \ref{thm} to prove that Conjecture \ref{conj1} holds for certain families of centerless groups $G$.

\vspace{2mm}

A group $G$ is said to be \emph{almost simple} if $A\subseteq G \subseteq \Aut(A)$ for some non-abelian simple group $A$, where $A$ is being identified with $\Inn(A)$. Since $A$ is centerless, so is $G$. It is known that $A$ is the socle of $G$, so every non-trivial normal subgroup of $G$ contains $A$, and that $\Aut(G)$ is also an almost simple group with socle isomorphic to $A$ (e.g see \cite[Lemmas 4.2 and 4.3]{Tsang ASG}).

\begin{thm}\label{thm1} The quotient $T(G)$ is cyclic of order $2$ for all almost simple groups $G$. 
\end{thm}
\begin{proof}Since $\Aut(G)$ is also almost simple, one of the $Q_1,Q_2$ in Theorem \ref{thm} must be trivial, for otherwise $Q_1\cap Q_2$ would contain the socle of $\Aut(G)$. It implies that the $K_1,K_2$ in Theorem \ref{thm} cannot be both proper.
 \end{proof}

A group $G$ is said to be \emph{perfect} if $G=[G,G]$.

\begin{thm}\label{thm2} The quotient $T(G)$ has exponent at most $2$ for all centerless perfect groups $G$.
\end{thm}
\begin{proof}Since $G$ is perfect, the $K_1,K_2$ in Theorem \ref{thm} satisfy $G = K_1K_2$, and so certainly $G/K_1K_2$ cannot have exponent greater than $2$.
\end{proof}

A group is said to be \emph{complete} if $G$ is centerless and $\Aut(G) = \Inn(G)$.

\begin{thm}\label{thm3}
The quotient $T(G)$ has exponent at most $2$ for all complete groups $G$.
\end{thm}
\begin{proof}Since $\Aut(G)=\Inn(G)$, trivially the $Q_1,Q_2$ in Theorem \ref{thm} have to lie inside $\Inn(G)$, so they cannot exist.
\end{proof}

Finally, we consider finite centerless groups whose orders are \emph{small}. Using our Theorem \ref{thm} and by running the code in the Appendix in  \textsc{Magma} \cite{magma}, we verified the following. The author thanks Dr. Qin Chao (Chris King) for running the code for her on his computer.

\begin{thm}\label{thm4}The quotient $T(G)$ has exponent at most $2$ for all centerless groups $G$ of order at most $2000$, except possibly when
\[|G| = 1536, \,\ G = \textsc{SmallGroup}(605,5),\,\ G = \textsc{SmallGroup}(1210,11).\]
\end{thm}

We had to exclude the order $1536$ because there are over $4\times 10^8$ groups of order $1536$ and the code just takes too long to run. Although many of these groups are not centerless, our code still needs to go through each group and test whether it is centerless. We were also unable to rule out the groups
\[\textsc{SmallGroup}(605,5),\,\ \textsc{SmallGroup}(1210,11)\]
using only the conditions in Theorem \ref{thm}. What is going on with these two groups is that, in the notation of Theorem \ref{main thm}, for non-trivial $f,h$ the kernels $\ker(f),\ker(h)$ are actually series-equivalent, so there is the possibility that
\[
 g(\ker(f)) = \ker(h),\,\ g(\ker(h)) = \ker(f).
 \]
Theorem \ref{main thm} is not strong enough to rule out this possibility and perhaps it actually occurs. To show that the class of the corresponding $\pi\in \NHol(G)$ in $T(G)$ has order $2$, one would need more information. 


\section*{Acknowledgment}

The author would like to thank the referee for helpful comments.

\newpage

\addresseshere

\begin{landscape}

\section*{Appendix: \textsc{Magma} code}

\begin{lstlisting}[
  mathescape,
  columns=fullflexible,
  basicstyle=\ttfamily,
]

IsCenterless:=function(X)
return #Center(X) eq 1;
end function; 

IsSeriesEquivalent:=function(G,X,Y)
return IsIsomorphic(X,Y) and IsIsomorphic(G/X,G/Y);
end function;

TestOrders:=[n:n in [2..2000]|not n eq 1536 and not IsPrimePower(n)];
//Groups of prime power order have non-trivial center and so might be ignored.
for n in TestOrders do
  for i in [1..NumberOfSmallGroups(n)] do
  G:=SmallGroup(n,i);
    if IsCenterless(G) then
    KSet:=[K`subgroup:K in NormalSubgroups(G)|
                      not K`order in [1,n] and IsCenterless(G/K`subgroup)];
    //The non-trivial proper normal subgroups of $G$ with centerless quotient.
    //A set of possibilities for the choice of $K$.
    AutG:=AutomorphismGroup(G);
    OutG:=[a:a in Generators(AutG)|not IsInner(a)];
    KSetChar:=[K:K in KSet|forall{a:a in OutG|a(K) eq K}];
    //The non-trivial proper characteristic subgroups of $G$ with centerless quotient.
    //A set of possibilities for the choices of $K_1$ and $K_2$.
    Comm:=CommutatorSubgroup(G);
    K12Set:=[<K1,K2>:K1,K2 in KSetChar|
                          #(K1 meet K2) eq 1
                      and Comm subset sub<G|K1,K2> 
                      and not Exponent(G/sub<G|K1,K2>) le 2
                      and exists{K:K in KSet|IsSeriesEquivalent(G,K,K1) and not K eq K1}];
    //A set of possibilities for the choice of $(K_1,K_2)$.
    QOrd:={n/#K[1]:K in K12Set} join {n/#K[2]:K in K12Set};
    //A set of possible orders of $Q_1\simeq G/K_1$ and $Q_2\simeq G/K_2$.
    Aut:=PermutationGroup(AutG);
    QSet:=[Q`subgroup:Q in NormalSubgroups(Aut:OrderDividing:=n)|
                      Q`order in QOrd and IsCenterless(Q`subgroup)];
    //A set of possibilities for the choices of $Q_1$ and $Q_2$.
    IsoToG:=[N`subgroup:N in NormalSubgroups(Aut:OrderEqual:=n)|IsIsomorphic(G,N`subgroup)];
    Q12Set:=[<Q1,Q2>:Q1,Q2 in QSet|
                          #(Q1 meet Q2) eq 1
                      and #[N:N in IsoToG|N subset sub<Aut|Q1,Q2>] ge 2]; 
    //A set of possibilities for the choice of $(Q_1,Q_2)$.
      if exists{<K[1],K[2],Q[1],Q[2]>:K in K12Set,Q in Q12Set|
                        IsIsomorphic(G/K[1],Q[1]) and IsIsomorphic(G/K[2],Q[2])} then
      <n,i>; //$\textsc{SmallGroup}$ ID of the groups $G$ that did not pass the test. 
      end if;
    end if;
  end for;
end for;

Output:
<605,5>
<1210,11>

\end{lstlisting}

\end{landscape}

\end{document}